\documentclass{amsart}
\usepackage[utf8]{inputenc}
\usepackage{colonequals}
\usepackage{color,amssymb,mathrsfs,amsmath,amsfonts,hyperref}

\hypersetup{
    colorlinks=true,
    linkcolor=blue,
    filecolor=magenta,      
    urlcolor=cyan,
} 

\usepackage[capitalise]{cleveref}
\crefformat{equation}{(#2#1#3)}
\crefrangeformat{equation}{(#3#1#4--#5#2#6)}
\crefformat{enumi}{(#2#1#3)}
\crefrangeformat{enumi}{(#3#1#4--#5#2#6)}

\newtheorem{theorem}{Theorem}[section]
\newtheorem*{theorem*}{Theorem}
\newtheorem{lemma}[theorem]{Lemma}

\newtheorem{corollary}[theorem]{Corollary}

\theoremstyle{definition}

\newtheorem{chunk}[theorem]{}

\theoremstyle{remark}
\newtheorem{remark}[theorem]{Remark}

\newcommand{\ZZ}{\mathbb{Z}}
\newcommand{\Spec}{\operatorname{Spec}}

\newcommand{\hh}{\operatorname{H}}

\newcommand{\RHom}{\mathsf{R}\mathrm{Hom}}
\newcommand{\m}{\mathfrak{m}}
\newcommand{\p}{\mathfrak{p}}

\newcommand{\lotimes}{\otimes^{\mathsf{L}}}
\newcommand{\V}{\operatorname{V}}

\newcommand{\llam}{\mathsf{L}\Lambda}
\newcommand{\rgam}{\mathsf{R}\Gamma}

\newcommand{\fp}{\mathfrak{p}}
\newcommand{\depth}{\operatorname{depth}}
\newcommand{\width}{\operatorname{width}}
\renewcommand{\leq}{\leqslant}
\renewcommand{\geq}{\geqslant}
\renewcommand{\sup}{\operatorname{sup}}
\renewcommand{\inf}{\operatorname{inf}}
\newcommand{\pdim}{\operatorname{proj.dim}}
\newcommand{\supp}{\operatorname{supp}}
\newcommand{\height}{\operatorname{ht}}
\newcommand{\lra}{\longrightarrow}
\numberwithin{equation}{theorem}

\title{\mbox{The Improved New Intersection Theorem revisited}}

\author[L.W.~Christensen]{Lars Winther Christensen}
\address{Texas Tech University, Lubbock, TX 79409, U.S.A.}
\email{lars.w.christensen@ttu.edu}
\email{lferraro@ttu.edu}

\author[L.~Ferraro]{Luigi Ferraro}

\thanks{L.W.C.\ was partly supported by Simons Foundation
  collaboration grant 428308.}

\date{19 October 2022}



\begin{document}

\maketitle

\begin{abstract}
  We prove a generalized version of Evans and Griffith's Improved New
  Intersection Theorem: Let $I$ be an ideal in a local ring $R$. If a
  finite free $R$-complex, concentrated in nonnegative degrees, has
  $I$-torsion homology in positive degrees, and the homology in degree
  0 has an $I$-torsion minimal generator, then the length of the
  complex is at least $\dim R-\dim R/I$. This improves the bound
  $\height I$ obtained by Avramov, Iyengar, and Neeman in 2018.
\end{abstract}

\section*{Introduction}

\noindent
In its various forms, the The New Intersection Theorem is concerned
with the length of a finite free complex, that is, a complex
\begin{equation*}
  F \colon 0 \lra F_n \lra \cdots \lra F_1 \lra F_0 \lra 0
\end{equation*}
of finitely generated free modules, over a local ring $(R,\m)$. The
classic version, due to Peskine and Szpiro, \cite{CPsLSz74} asserts
that if $\hh(F)$ is non-zero and each homology module $\hh_i(F)$ is of
finite length, then $n \geq \dim R$ holds.  The statement known as the
Improved New Intersection Theorem was first established within the
proof of Evans and Griffith's Syzygy Theorem \cite{GEEPGr81}. Hochster
states it in \cite{MHc83} as follows:
\begin{enumerate}
\item If the homology modules $\hh_i(F)$ for $i >0$ have finite length
  and a nonzero minimal generator of $\hh_0(F)$ generates a submodule
  of finite length, then $n \geq \dim R$ holds.
\end{enumerate}
A slightly stronger statement was obtained by Iyengar \cite[Theorem
3.1]{SIn99}:
\begin{enumerate}\setcounter{enumi}{1}
\item If the modules $\hh_i(F)$ for $i >0$ are of finite length and an
  ideal $I$ annihilates a nonzero minimal generator of $\hh_0(F)$,
  then $n \geq \dim R - \dim R/I$ holds.
\end{enumerate}
We notice that under the assumptions in $(1)$, some power of $\m$
annihilates $\hh_i(F)$ for all $i > 0$ as well as a minimal generator
of $\hh_0(F)$. The original statements in
\cite{GEEPGr81,SIn99,CPsLSz74} were made for equicharacteristic
rings. The New Intersection Theorem was proved in mixed
characteristics by Roberts \cite{rob} and, through the work of Andr\'e
\cite{YAn18}, the remaining statements are now also known to hold for
all local rings.

The original Improved New Intersection Theorem was generalized by
Avramov, Iyengar, and Neeman \cite{AIN-18} as follows:
\begin{enumerate}\setcounter{enumi}{2}
\item If an ideal $I$ annihilates the homology modules $\hh_i(F)$ for
  $i >0$ as well as a nonzero minimal generator of $\hh_0(M)$, then
  $n \geq \height I$ holds.
\end{enumerate}
One always has $\dim R - \dim R/I \geq \height I$, so the bound in (3)
is weaker than the bound in (2), but so are the assumptions on
$\hh(F)$. The main result of this paper is a common generalization of
these last two statements:
\begin{enumerate}\setcounter{enumi}{3}
\item If an ideal $I$ annihilates the homology modules $\hh_i(F)$ for
  $i >0$ as well as a nonzero minimal generator of $\hh_0(F)$, then
  $n \geq \dim R - \dim R/I$ holds.
\end{enumerate}
If the homology modules $\hh_i(F)$ for $i > 0$ are $I$-torsion, then
they are all annihilated by some fixed power $I^n$ and one has
$\dim R/I^n = \dim R/I$, so (4) is equivalent to the statement made in
the abstract. Finally, we notice that for an $\m$-primary ideal $I$
the statements $(2)$--$(4)$ reduce to the original statement $(1)$.

The work of Andr\'e mentioned above proved the existence of big
Cohen-Macaulay modules over any local ring, and it had already been
established that the existence of such modules was sufficient to prove
the Improved New Intersection Theorem, see Hochster \cite{MHc75a} and
Iyengar \cite{SIn99}. The proof of our main result, which is
\Cref{thm:minit}, is inspired by a more recent proof of $(2)$ by
Iyengar, Ma, Schwede, and Walker \cite{IMSW-22}. Our twist comes down
to controlling the depth of derived $\m$-complete complexes.
 
\section{Derived complete complexes}
\noindent
Throughout the paper, $R$ is a commutative noetherian local ring with
unique maximal ideal $\m$ and residue field $k$. For a finitely
generated $R$-module $M$ and a prime ideal $\p$, Bass' \cite[Lemma
(3.1)]{HBs63} yields the inequality
$\depth_R M \le \depth_{R_\p}M_\p + \dim R/\p$. The main result of
this section, which is key to our proof of Theorem \ref{thm:minit}, is
that the same inequality holds for derived $\m$-complete
$R$-complexes.

\begin{chunk}
  We use homological notation, i.e.\ lower indexing, for
  $R$-complexes. For an $R$-complex $M$, the homological supremum and
  infimum are
  \begin{equation*}
    \sup M = \sup\{n \in \ZZ \mid \hh_n(M) \ne 0\} \quad\text{and}\quad 
    \inf M = \inf\{n \in \ZZ \mid \hh_n(M) \ne 0\} \:.
  \end{equation*}
\end{chunk}

\begin{chunk}\label{chk:Prelim}
  Let $I$ be an ideal in $R$. As is standard we denote the right
  derived $I$-torsion functor by $\rgam_I$ and the left derived
  $I$-completion functor by $\llam^I$. They are adjoint functors, see
  Alonso Tarr\'io, Jerem\'ias Lopez, and Lipman \cite[Theorem
  (0.3)]{AJL-97}, and an $R$-complex $M$ is called derived $I$-torsion
  or derived $I$-complete if it is isomorphic in the derived category
  to $\rgam_I(M)$ or $\llam^I(M)$, respectively. For an $R$-complex
  $M$ the vanishing of local (co)homology, i.e.\ $\hh(\rgam_I{M})$ and
  $\hh(\llam^I{M})$ detects, or if one wishes defines, the depth and
  width invariants relative to $I$:
  \begin{gather}
    \depth_R(I,M) = -\sup\RHom_R(R/I,M) = -\sup{\rgam_I(M)}\label{eq:SupRGam} \\
    \width_R(I,M) = \inf(R/I\lotimes_RM) =
    \inf{\llam^I(M)}\label{eq:InfLLam} \:;
  \end{gather}
  see Foxby and Iyengar \cite[Theorem 2.1 and Theorem
  4.1]{HBFSIn03}. From these equalities and standard inequalities
  regarding homological suprema and infima, see Foxby \cite[Lemma
  2.1]{HBF77b}, one gets:
  \begin{align}
    \depth(I,M) \geq -\sup{M} \text{ with equality if and only if } \Gamma_I(\hh_{\sup{M}}(M)) \ne 0 \:.\label{eq:IDepth}\\
    \width(I,M) \geq \inf{M} \text{ with equality if and only if } \Lambda^I(\hh_{\inf{M}}(M)) \ne 0 \:.
  \end{align}
  Vanishing of local cohomology supported at the maximal ideal also
  detects the dimension of a finitely generated $R$-module:
  \begin{equation}\label{eq:InfRGam}
    \dim_RM = -\inf\rgam_\m(M) \:.
  \end{equation}
\end{chunk}

The next lemma is folklore---Foxby and Iyengar allude to it in the
text preceding \cite[Proposition 2.2]{HBFSIn03}---but we didn't find a
reference to cite.

\begin{lemma}
  \label{lem:ex14.3.3}
  Let $I$ be an ideal in $R$ and $M$ and $N$ be $R$-complexes. If $M$
  is derived $I$-torsion with $\hh{(M)}$ nonzero and bounded below,
  then the next inequalities hold
  \begin{align}
    \tag{a}
    \inf(M\lotimes_RN) & \geq \inf M+ \width_R(I,N) \:.\\
    \tag{b}
    -\sup\RHom_R(M,N) & \geq \inf{M} + \depth_R(I,N) \:.
  \end{align}
\end{lemma}

\begin{proof}
  To prove the inequality (a), we first observe that there are
  isomorphisms in the derived category as follows:
  \begin{align*}
    M\lotimes_RN &\simeq \rgam_I M\lotimes_RN \\
                 &\simeq M\lotimes_R\rgam_I N \\
                 &\simeq M\lotimes_R\rgam_I\llam^IN \\
                 &\simeq\rgam_IM\lotimes_R\llam^IN \\
                 &\simeq M\lotimes_R\llam^IN \:.
  \end{align*}
  Indeed, the first and last isomorphisms hold as $M$ is derived
  $I$-torsion, the second and fourth isomorphism follow from
  \cite[(2.1))]{AJL-97}, and the third isomorphism follows from
  \cite[Corollary (5.1.1)]{AJL-97}.  This justifies the first equality
  in the following chain of (in)equalities
  \begin{align*}
    \inf(M\lotimes_RN)&=\inf(M\lotimes_R\llam^I N) \\
                      &\geq \inf M+\inf\llam^I N \\
                      &=\inf M+\width_R(I,N) \:;
  \end{align*}
  here the inequality holds by \cite[Lemma 2.1]{HBF77b} and
  \eqref{eq:InfLLam} yields the last equality.

  To prove the inequality (b), we first observe that the following
  chain of isomorphisms in the derived category holds:
  \begin{align*}
    \RHom_R(M,N)&\simeq\RHom_R(\rgam_I M,N) \\
                &\simeq\RHom_R(M,\llam^I N) \\
                &\simeq\RHom_R(M,\llam^I\rgam_I N) \\
                &\simeq\RHom_R(\rgam_I M,\rgam_I N) \\
                &\simeq\RHom_R(M,\rgam_I N) \:.
  \end{align*}
  Indeed, the first and last isomorphisms hold as $M$ is derived
  $I$-torsion, the second and fourth isomorphisms follows from the
  fact that $\rgam_I$ and $\llam^I$ are adjoint functors, see
  \cite[Theorem (0.3)]{AJL-97}, and the third isomorphism follows from
  \cite[Corollary (5.1.1)]{AJL-97}.  This explains the first equality
  in the following chain of (in)equalities:
  \begin{align*}
    -\sup\RHom_R(M,N)&=-\sup\RHom_R(M,\rgam_I N) \\
                     &\geq \inf M-\sup\rgam_I N \\
                     &=\inf M + \depth_R(I,N) \:;
  \end{align*}
  here the inequality holds by \cite[Lemma 2.1]{HBF77b}, and
  \eqref{eq:InfLLam} yields the last equality.
\end{proof}

\begin{theorem}\label{prop:bass}
  Let $M$ be a derived $\m$-complete $R$-complex. For every prime
  ideal $\p$ in $R$ there is an inequality
  \begin{equation*}
    \depth_RM\leq\depth_{R_\p}M_\p+\dim R/\p.
  \end{equation*}
\end{theorem}

\begin{proof}
  The claim follows from the following chain of (in)equalities
  \begin{align*}
    \depth_{R_\p}M_\p &\geq \depth_R(\p,M) \\
                      &=-\sup\RHom_R(R/\p,M)\\
                      &=-\sup\RHom_R(R/\p,\llam^\m M) \\
                      &=-\sup\RHom_R(\rgam_\m (R/\p),M) \\
                      &\geq \inf\rgam_\m (R/\p) + \depth_RM \\
                      &=\depth_RM-\dim R/\p \:,
  \end{align*}
  where the first inequality holds by \cite[Proposition
  2.10]{HBFSIn03}, the first equality is part of \Cref{eq:SupRGam},
  the second equality follows from the hypothesis that $M$ is derived
  $\m$-complete, the third equality holds as $\llam$ and $\rgam$ are
  adjoint functors, the last inequality holds by \Cref{lem:ex14.3.3},
  and \eqref{eq:InfRGam} yields the last equality.
\end{proof}

\begin{corollary}
  \label{cor:bass}
  Let $M$ be a derived $\m$-complete $R$-complex. For every ideal $I$
  in $R$ there is an inequality
  \begin{equation*}
    \depth_RM\leq\depth_R(I,M)+\dim R/I \:.
  \end{equation*}
\end{corollary}

\begin{proof}
  From \cite[Proposition 2.10]{HBFSIn03} one gets
  \begin{equation*}
    \depth_R(I,M)=\inf\{\depth_{R_\p}M_\p\mid \p\in \V(I)\} \:.
  \end{equation*}
  Therefore, $\depth_R(I,M)=\depth_{R_\p}M_\p$ holds for some
  $\p\in \V(I)$, and now the asserted inequality follows from
  \Cref{prop:bass} as $\dim R/\p \le \dim R/I$ holds.
\end{proof}

Notice that for a prime ideal $I$ the inequality in \Cref{cor:bass}
may be stronger than the inequality in \Cref{prop:bass}.

\section{An Improved New Intersection Theorem}
\label{sec:main}

\noindent We now get to the main result of the paper.

\begin{chunk}
  We recall from \cite{IMSW-22} that an $R$-complex of maximal depth
  is a complex $M$ satisfying the following three conditions:
  \begin{enumerate}
  \item $\hh(M)$ is bounded;
  \item The canonical map $\hh_0(M)\rightarrow\hh_0(k\lotimes_RM)$ is
    nonzero;
  \item $\depth_RM=\dim R$.
  \end{enumerate}
\end{chunk}

The obvious example of a complex of maximal depth is a big
Cohen-Macaulay module, and such modules exist over every local
ring. The interest in complexes derives from the fact that homological
conjectures---The Canonical Element Conjecture to be specific---in the
presence of a dualizing complex implies the existence of complexes of
maximal depth with degreewise finitely generated homology, see
\cite[Remarks 4.7 and 4.15]{IMSW-22}.

\begin{theorem}
  \label{thm:minit}
  Let $I$ be an ideal in $R$ and
  \begin{equation*}
    F\colon0\longrightarrow F_n\longrightarrow\cdots\longrightarrow
    F_1 \longrightarrow F_0\longrightarrow 0
  \end{equation*}
  a finite free $R$-complex with $\hh_0(F) \ne 0$. If $\hh_i(F)$ is
  $I$-torsion for $i>0$ and a minimal generator of $\hh_0(F)$ is
  $I$-torsion, then $n\geq \dim R-\dim R/I$ holds.
\end{theorem}

\begin{proof}
  Let $M$ be a derived $\m$-complete complex of maximal depth; such a
  complex exists by \cite[Lemma 3.4]{IMSW-22}. Let $s$ be the integer
  $\sup (F\otimes_RM)$ and notice from \cite[Lemma 3.1]{IMSW-22} that
  one has that $s\geq0$.

  Let $\p$ be in $\mathrm{Ass}_R\mathrm{H}_s(F\otimes_R M)$. It
  follows that $\mathrm{H}(F\otimes_RM)_\p$ is nonzero and, hence,
  $\mathrm{H}(F)_\p$ and $\mathrm{H}(M)_\p$ are nonzero as well. We
  have the following chain of (in)equalities
  \begin{equation}\label{eq:pdim}
    \begin{aligned}
      \pdim_{R_\p}F_\p & = \depth_{R_\p}M_\p-\depth_{R_\p}(F\otimes_RM)_\p \\
      & = \depth_{R_\p}M_\p+s \\
      & \geq \depth_RM-\dim R/\p+s\\
      & = \dim R-\dim R/\p+s \:,
    \end{aligned}
  \end{equation}
  where the first equality is the Auslander-Buchsbaum equality, see
  \cite[Theorem 2.4]{HBFSIn03}, the second equality follows from
  \eqref{eq:IDepth}, the inequality holds by \Cref{prop:bass}, and the
  last equality holds as $M$ is a complex of maximal depth.

  Assume first that $s\geq 1$ holds. In this case it suffices to show
  that $I$ is contained in $\p$ as one then has,
  \begin{equation*}
    n \geq \pdim_RF \geq \pdim_{R_\p}F_\p\geq \dim R-\dim R/\p+s > \dim
    R-\dim R/I \:.
  \end{equation*}
  To see that $\p$ contains $I$, assume towards a contradiction that
  $I\not\subseteq \p$. It follows that $F_\p$ is isomorphic to
  $\mathrm{H}_0(F)_\p$ in the derived category, as $\mathrm{H}_i(F)$
  is $I$-torsion for $i\geq1$ and, therefore, $\sup F_\p=0$. One now
  has the following chain of (in)equalities
  \begin{align*}
    \depth R_\p&=\depth F_\p+\pdim F_\p\\
               &\geq \pdim F_\p\\
               &\geq \dim R-\dim R/\p+s\\
               &\geq \dim R_\p+s \:,
  \end{align*}
  which is absurd as $s$ is positive. The equality in the display
  above is the Auslander-Buchsbaum equality, see \cite[Theorem
  2.4]{HBFSIn03}, the first inequality is trivial,  the second
   follows from \eqref{eq:pdim}, and the last inequality is
  standard.

  It remains to consider the case $s=0$. It follows from the finite
  generation of $\hh_0(F)$, Nakayama's Lemma, and \cite[Lemma
  3.1]{IMSW-22} that each minimal generator of $\hh_0(F)$ gives rise
  to a nonzero element in $\hh_0(F\otimes_R M)$. Thus, by hypothesis,
  there is an $I$-torsion element of $\hh_0(F\otimes_RM)$, i.e.\
  $\Gamma_I(\hh_0(F\otimes_RM))\neq0$ and, therefore,
  $$\depth_R(I,F\otimes_RM)=-\sup(F\otimes_RM)=-s=0$$     by \eqref{eq:IDepth}. By \cite[(2.2)]{IMSW-22}, the complex $F\otimes_RM$ is derived $\m$-complete, therefore, applying
  \Cref{cor:bass}, one gets
  \begin{equation*}
    \depth_R(F\otimes_RM) \leq \dim R/I \:.
  \end{equation*}
  It remains to apply the Auslander-Buchsbaum equality:
  \begin{equation*}
    \pdim_RF = \depth_RM-\depth_R(F\otimes_RM)\geq \dim R-\dim R/I
    \:. \qedhere
  \end{equation*}
\end{proof}
\begin{chunk}
  Recall from Foxby \cite{HBF79} that for an $R$-complex $M$ the small
  support is the set
  $$\supp_RM = \{\p\in\Spec R \mid \hh(M \lotimes_R k(\p)) \ne 0\}
  \:,$$ where $k(\fp)$ denotes the residue field of the local ring
  $R_\p$.
\end{chunk}

\begin{remark}
  \label{rmk:ineq-inf}
  Let $M$ be an $R$-complex with bounded homology and $\p$ a prime
  ideal in $\supp_RM$. There are inequalities,
  \begin{equation}
    \label{eq:ineq-inf}
    \depth_{R_\p}M_\p \le \dim_{R_\p}M_\p \le \dim R_\p - \inf{M_\p} \le  \dim R - \dim R/\p - \inf{M_\p} \:;
  \end{equation}
  indeed, the first inequality holds by \cite[Corollary 3.9]{HBF79},
  the second inequality follows from the definition of dimension of
  complexes, also from \cite{HBF79}, and the third is standard. Thus,
  for a derived $\m$-complete $R$-complex $M$ of maximal depth it
  follows from \Cref{prop:bass} and \eqref{eq:ineq-inf} that the
  inequalities
  \begin{equation*}
    \dim R-\dim R/\p\leq\depth_{R_\p}M_\p\leq\dim R-\dim R/\p-\inf
    M_\p,
  \end{equation*}
  hold for every prime ideal $\p$ in $\supp_RM$.
\end{remark}

\begin{theorem}
  \label{thm:module}
  Let $M$ be a derived $\m$-complete $R$-module of maximal depth. For
  every prime ideal $\p$ in $\supp_RM$ the equality
  $\dim R = \dim R_\fp + \dim R/\fp$ holds and $M_\p$ is an
  $R_\p$-module of maximal depth.
\end{theorem}

\begin{proof}
  Let $\p$ be a prime ideal in $\supp_RM$; as $\p$ in particular is in
  the support of $M$, the inequalities \eqref{eq:ineq-inf} read
  $\depth_{R_\p}M_\p \le \dim R_\p \le \dim R - \dim R/\p$. The
  inequality from \Cref{prop:bass} can be rewritten as
  $\dim R - \dim R/\p \le \depth_{R_\p}M_\p$. Combining these
  inequalities one gets the equality $\dim R = \dim R_\p + \dim R/\p$
  as well as $\depth_{R_\p}M_\p = \dim R_\p$. It remains to see that
  $M_\p \otimes_{R_\p} k(\p)$ is non-zero. Set $d = \dim R_\p$ and let
  $K$ be the Koszul complex on a sequence
  $\mathbf{x} = x_1,\ldots,x_d$ of parameters for $R_\fp$. By
  \cite[Definitions 2.3 and 4.3]{HBFSIn03} one has
  $\depth_{R_\p}M_\p = d - \sup{(K \otimes_{R_\p} M_\p)}$ and
  $\width_{R_\p}M_\p = \inf{(K \otimes_{R_\p} M_\p)}$. Therefore, one
  has
  $$\depth_{R_\p} M_\p + \width_{R_\p} M_\p = d - \sup{(K
    \otimes_{R_\p} M_\p)} + \inf{(K \otimes_{R_\p} M_\p)} \le d \:.$$
  This forces $\width_{R_\p}M_\p = 0$, whence
  $M_\p \otimes_{R_\p} k(\p) \ne 0$ by \eqref{eq:InfLLam} as
  desired. (The inequality displayed above is \cite[Corollary
  6.1.10]{JRS90}, and Strooker credits Bartijn with observing it in
  his thesis.)
\end{proof}

\begin{remark}
  The $\m$-adic completion of a big Cohen-Macaulay module is an
  example of a derived $\m$-complete module of maximal depth. Indeed,
  such a module is a balanced big Cohen-Macaulay module, see Bruns and
  Herzog \cite[Corollary 8.5.3]{bruher}, and derived $\m$-complete,
  see for example Schenzel and Simon \cite[Proposition 2.5.7(a),
  Example 7.3.2(d), and Theorem 7.5.13(a)]{PScAMS}. For such modules
  the equality in \Cref{thm:module} was proved by Sharp \cite[Theorem
  3.2]{RYS81}, who called the set $\supp_RM$ the supersupport of $M$.
\end{remark}

\providecommand{\MR}[1]{\mbox{\href{http://www.ams.org/mathscinet-getitem?mr=#1}{#1}}}
\renewcommand{\MR}[1]{\mbox{\href{http://www.ams.org/mathscinet-getitem?mr=#1}{#1}}}
\providecommand{\arxiv}[2][AC]{\mbox{\href{http://arxiv.org/abs/#2}{\sf
      arXiv:#2 [math.#1]}}} \def\cprime{$'$}
\providecommand{\bysame}{\leavevmode\hbox to3em{\hrulefill}\thinspace}
\providecommand{\MR}{\relax\ifhmode\unskip\space\fi MR }
\providecommand{\MRhref}[2]{%
  \href{http://www.ams.org/mathscinet-getitem?mr=#1}{#2} }

\end{document}